\documentclass[12pt]{article}
\usepackage{amssymb}
\usepackage{amsthm}
\usepackage{amsmath}
\usepackage{graphicx}
 \newtheorem{thm}{Theorem}[section]
 
 \newtheorem{lem}[thm]{Lemma}
 
 \theoremstyle{definition}

 \numberwithin{equation}{section}
 \numberwithin{equation}{section}

\newtheorem{Lemma A.1}{Lemma A.1}
\theoremstyle{definition}

\theoremstyle{remark}

\def\R{\bf R}

\def\al{\aligned}
\def\eal{\endaligned}

\def\be{\begin{equation}}
\def\ee{\end{equation}}
\def\lab{\label}

\def\e{\epsilon}
\def\lam{\lambda}

\def\R{\bf R}

\def\al{\aligned}

\def\p{\partial}
\numberwithin{equation}{section}

\begin{document}
\title{ On ancient periodic solutions to  the axi-symmetric Navier-Stokes equations }

\author{Zhen Lei \footnotemark[1]
\and Xiao Ren   \footnotemark[1]
\and Qi S. Zhang    \footnotemark[2]
}
\renewcommand{\thefootnote}{\fnsymbol{footnote}}
\footnotetext[1]{
School of Mathematical Sciences; LMNS and Shanghai Key Laboratory for Contemporary Applied Mathematics,
Fudan University, Shanghai 200433, P. R.China.}
\footnotetext[2]{Department of
Mathematics,  University of California, Riverside, CA 92521, USA.}

\date{\today}

\maketitle

\begin{abstract}

It has been an old and challenging problem to classify bounded ancient solutions of the incompressible Navier-Stokes equations, which could play a crucial role in the study of global regularity theory.
In the  works \cite{KNSS, SS09}, the authors made the following conjecture:
 \textit{for the 3D
axially symmetric Navier Stokes equations, bounded mild ancient solutions are constants}. In this article, we solve this conjecture in the case that $u$ is periodic in $z$. To the best of our knowledge, this seems to be the first result on this conjecture without unverified decay conditions. It also shows that nontrivial periodic solutions are not models of possible singularities  or
high velocity regions. Some partial results in the non-periodic case is also given.

\end{abstract}

\section{Introduction}

The classical Liouville theorem, stating that bounded harmonic functions in $\R^n$ are constants, has been extended to many other elliptic and parabolic equations in different settings, and further,  numerous applications have been found. The following is one of them. To study singularity formation for solutions of  nonlinear equations, one often blows up the solution near singularity. This procedure often results in a bounded solution which exists in whole space or, in the case of evolution equations, one which exists in the time interval $(-\infty, 0]$. Such solutions are often referred to as ancient solutions. Information about ancient solutions reveals singularity structure or the lack of singularity of the original solutions, as well as the behavior of solutions in regions with high value. For the 3 dimensional incompressible Navier-Stokes equations, one can also carry out this procedure. Naturally, one may hope to classify the resulting ancient solutions.  However, this problem seems beyond the reach of existing theories. In fact, it is still widely open even for the stationary case. The full stationary problem seems intractable since it contains, as a special case, another old unsolved problem concerning D-solutions, which asks whether a 3 dimensional stationary solution with finite Dirichlet energy and vanishing at infinity is identically 0. Although at first glance,
this extra condition seems very restrictive compared with boundedness of solutions, it has not offered any help, even in the axially symmetric case. See for instance recent papers \cite{CPZZ}, \cite{Ser}.

During the last few decades, a large amount of analysis has been carried out for the 3 dimensional incompressible Navier-Stokes equations. We mention here a few well-known and related works. In 1934, Jean Leray\cite{Le} raised the existence problem of backward self-similar solutions which can be considered as ancient solutions with a uniform profile.  In \cite{NRS}, J. Necas, M. Ruzicka and V. ${\rm \breve{S}}$ver${\rm \acute{a}}$k proved that such solutions must be trivial if the profile is in $L^3$(see \cite{Ts} for a local version). In the significant work \cite{ESS}, the authors showed that $L_t^\infty L_x^3$ solutions must be regular. The first step of their proof is to rescale the solution near potential singularities and obtain a bounded ancient solution as described in the first paragraph. Both of the two works above are based on the landmark partial regularity theory of Caffarelli-Kohn-Nirenberg\cite{CKN}. More precisely, the authors showed that the $1$-dimensional Hausdorff measure of the singularity set of suitable weak solutions must be $0$. In particular, this implies that for the axially symmetric case, blow-up can only happen along the axis. We remark that the Caffarelli-Kohn-Nirenberg method is a perturbative one. In many cases, it seems that the blow-up method has its own power, compared with the perturbative techniques.

In \cite{KNSS}, G. Koch, N. Nadirashvili, G. Seregin and V. Sverak made the following conjecture: \textit{for incompressible
axially symmetric Navier-Stokes equations (ASNS) in three dimensions: bounded mild ancient solutions are constants}. They obtained some partial results under critical decay conditions for the velocity. For example, Liouville theorem for ASNS holds if one assumes that $|v(x, t)| \le C/|x'|$ where $v$ is the velocity, $x=(x^1, x^2, x^3)$ and $x'=(x^1, x^2, 0)$ are the Cartesian coordinates. See also an extension in \cite{LZ11} to the case that $v$ is in $BMO^{-1}$ (which is ture if $|v_z| \le
\frac{C}{|x'|}$) and further improvements in \cite{SZ}.

In this work, we consider bounded ancient solutions of ASNS, which are periodic in the $z$ variable of the cylindrical system. In this case the conjecture is fully solved. We emphasize that no decay condition is imposed on the velocity. It is helpful to compare our result with another nonlinear parabolic system, namely, the 3 dimensional Ricci flow. Perelman \cite{Pe} showed that the typical model for high curvature region is $S^2 \times R$ which is periodic in the $S^2$ part.  In contrast, Theorem \ref{thzhouqi} shows periodic solutions are not models of high velocity region for ASNS, and appears to be the first result for Navier-Stokes equations in this direction.

Now let us elaborate the main results in detail. Let $v$ be the velocity; while $v_r$, $v_z$ and $v_\theta$ be the components of $v$ in the cylindrical coordinates of $\{ e_r, \, e_z, \, e_\theta \}$ respectively. Suppose $v_r$, $v_z$ and $v_\theta$ are independent of $\theta$, then ASNS takes the form of
\begin{align}
\label{eqasns}
\begin{cases}
   \big (\Delta-\frac{1}{r^2} \big )
v_r-(b\cdot\nabla)v_r+\frac{v_{\theta}^2}{r}-\frac{\partial
p}{\partial r}-\frac{\partial v_r}{\partial t}=0,\\
   \big   (\Delta-\frac{1}{r^2}  \big
)v_{\theta}-(b\cdot\nabla)v_{\theta}-\frac{v_{\theta}v_r}{r}-\frac
{\partial v_{\theta}}{\partial t}=0,\\
 \Delta v_z-(b\cdot\nabla)v_z-\frac{\partial p} {\partial
z}-\frac{\partial v_z}{\partial t}=0,\\
 \frac{1}{r}\frac{\partial (rv_r)}{\partial r} +\frac{\partial
v_z}{\partial z}=0,
\end{cases}
\end{align}
 where
\begin{equation} \label{eq-b}
 b(x,t) =v_r e_r + v_z e_z
\end{equation}
and the last equation is the divergence-free
 condition. Here, $\Delta$ is the cylindrical  scalar Laplacian
and $\nabla$ is the cylindrical gradient field:
 \begin{align*}
 \Delta =\frac{\partial^2}{\partial r^2}+
\frac{1}{r}\frac{\partial}{\partial r}+\frac{1}{r^2}\frac{\partial
^2}{\partial \theta ^2}+\frac{\partial^2}{\partial
 z^2},\ \
 \nabla =  \Big  (\frac{\partial }{\partial r},\frac{1}{r}\frac{\partial
}{\partial \theta},\frac{\partial }{\partial  z}  \Big  ).
\end{align*}
Observe that the equation for $v_{\theta}$ does not depend on the
pressure. Let $\Gamma=rv_{\theta}$, one sees that the function
$\Gamma$ satisfies
\begin{equation}\label{Gamma/vtheta}
\Delta \Gamma -(b\cdot\nabla)\Gamma-\frac{2}{r}\frac{\partial
\Gamma}{\partial r}-\frac{\partial\Gamma}{\partial t}=0,\,\text{div} \,  b=0.
\end{equation}

The main result of the paper is:

\begin{thm}
\label{thzhouqi}
Let $v=v_\theta e_\theta +v_r e_r + v_z e_z $ be a bounded mild ancient solution to the ASNS such that $\Gamma= r v_\theta$ is bounded. Suppose  $v$ is periodic in the $z$ variable. Then $v= c e_z$ where $c$ is a constant.
\end{thm}

For the definition of mild solutions, one can consult the paper \cite{KNSS}. Roughly speaking, these solutions satisfy certain integral equations involving the heat and Stokes kernel, ruling out the so-called parasitic solutions such as $v=a(t)$ and $P=- a'(t) \cdot x$ where $a(t)$ is any vector field depending only on time. Moreover, these solutions are smooth in space-time if they are bounded.

A remark is in order for the assumption that the function $\Gamma = r v_\theta$ is bounded.
It is well known that $\Gamma$ satisfies the maximum principle. Hence if the initial data of a Cauchy problem satisfies the condition, then it will be satisfied for all time. Since $\Gamma$ is scaling invariant, ancient solutions arising from the blow-up procedure will also satisfy the condition, with perhaps a different constant due to the shift of the axis during the blow-up. So this condition is essentially  not a restriction for the study of possibile singularities of ASNS.

Let us describe the general idea of the proof.  We will prove, by the De Giorgi-Nash-Moser method that $\Gamma$ satisfies a partially scaling invariant H\"older estimate which forces $\Gamma \equiv 0$.
Then the problem is reduced to the swirl free case that is solved in \cite{KNSS}.
In general this method will break down in large scale, unless one imposes scaling invariant decay conditions on $v_r$ and $v_z$.  Although no decay conditions on $v_r$ or $v_z$ are assumed in our theorem, in Section 2 and Section 3 we will demonstrate that the classical Nash-Moser iteration method can be carefully adapted to our situation. Our key observation is the following: the incompressibility condition $\nabla \cdot b=0$ along with the periodicity in $z$ gives us extra information on $v_r$. In fact we will essentially use  $v_r(r,\theta,z) = - \partial_z (L_\theta(r,\theta,z)-L_\theta(r,\theta,0)) \in (L^\infty)^{-1}$, where $L_\theta$ is the angular stream function(See \eqref{L-infty-(-1)}). Another helpful factor is that the spatial domain
$\mathbb{R}^2 \times S^1$ behaves like a 2 dimensional Euclidean space in large scale, even though it really behaves 3 dimensionally near the axis.

Next, we present a theorem  that deals with non-periodic case, under an extra condition that $\Gamma$ converges to its maximum at certain speed. Even though our method cannot yet reach  the full conjecture in \cite{KNSS}, it can be regarded as a step forward to prove the conjecture. Besides, the method may be of independent value and use elsewhere. In section 5 we will apply it to present a new proof for the steady periodic case.

Let
\be
\lab{limsupgam}
\lim\sup_{r \to \infty} \Gamma = \lim\sup_{r \to \infty} \sup_{z, t} \Gamma(r, z, t).
\ee  It will be shown in Section 4 that if $v$ is any bounded ancient solution such that $\Gamma$ is bounded, then $\lim\sup_{r \to \infty} \Gamma = \sup \Gamma$.

\begin{thm}
\lab{thgudai} Let $v=v_\theta e_\theta +v_r e_r + v_z e_z $ be a bounded mild ancient solution to the ASNS such that $\Gamma= r v_\theta$ is bounded. There exists a small number $\e_0 \in (0, 1)$, depending only on $\Vert v \Vert_\infty$,
 such that if
\be
\lab{giginfty2}
| \Gamma^2(r, z, t) - \lim\sup_{r \to \infty} \Gamma^2| \le \frac{\e_0}{r} \lim\sup_{r \to \infty} \Gamma^2
\ee
holds uniformly for $z$, $t$, and large $r$, then $v= c e_z$ where $c$ is a constant.

\end{thm}

The rest of the paper is organized as follows. In Section 2 we prove a mean value inequality using Moser's iteration with adaptations, where a dimension reduction effect is achieved. In Section 3, we use De Giorgi-Nash-Moser type arguments to prove Theorem \ref{thzhouqi}, following the scheme in \cite{LZ11}, which, in turn, builds on the idea of \cite{CSTY1}
and \cite{Z}. As mentioned earlier, the main idea is to use periodicity to overcome the lack of critical estimate for $v_r$ and $v_z$.

In Section 4, we will prove  Theorem \ref{thgudai}. We will apply a weighted energy method for the function
$\Gamma= r v_\theta$, exploiting the special structure of the equation and the fact that $\Gamma=0$ at the $z$ axis. It is well known that the usual energy method will run into the difficulty of
insufficient decay of solutions. The new idea of the proof lies in the construction of a special weight function, part of which is constructed explicitly and part of which is constructed by solving an auxiliary PDE. Besides, instead of $rdrdz$, we will perform energy estimates against the measure $\lambda(r)drdz$ with carefully chosen cut-off functions, which enables us to take the advantage of both the three dimensional behavior of the system near the axis and the two dimensional nature away from the axis.

In the last section, we revisit periodic ancient solutions, and focus on the steady ones. We use the method developed in Section 4 to reprove Liouville theorem for such solutions. We will also need a new observation that for $z$ periodic ancient solutions, $v_r$ converges to $0$ uniformly as $r \to \infty$.

\maketitle

\section{Local maximum estimate}
We denote $x=(r,\theta,z)\in \mathbb{R}^3$ in cylindrical coordinates. For $R>0$, we write
\[
D_R = \{x \in \mathbb{R}^3 \, |\, 0\le r <R, \theta \in \mathbb{S}^1, 0\le z<Z_0\}
\] and $P_R=D_R \times (-R^2,0]$ through out the paper. We emphasize that our choice of the cut-off function in the proof, together with the periodicity of solutions, helps us gain the crucial effect of dimension reduction.

\begin{lem}\label{lemma-21}
Assume that $\Phi\ge 0$ is a (Lipschitz) subsolution  to \eqref{Gamma/vtheta} in $\mathbb{R}^3\times (-\infty,0]$, with $b$ as in \eqref{eq-b} and bounded, i.e. $\Phi$ satisfies
\begin{equation} \label{eq-Phi}
\partial_t \Phi - \Delta \Phi + \frac{2}{r} \partial_r \Phi + b \cdot \nabla \Phi\le0.
\end{equation}
Also assume that $\Phi$ has period $Z_0$ in the $z$-direction and $\Phi\big|_{r=0}=0$. Then for any $R\ge1$, we have
$$\sup_{P_{R/2}} |\Phi | \le C \left\{ \frac{1}{R^4} \iint_{P_R} |\Phi|^2 dx ds \right\}^{\frac12},$$
where the constant $C$ does not depend on $R$.
\end{lem}
\begin{proof}  The main point of the lemma is that $|\Phi|$ is bounded by a constant multiple of its average
in $P_R$ for all large $R$. Due to the drift term in the equation, this is not obvious when $R$ approaches infinity. On the other hand, it is worth noticing when $R$ approaches 0, our proof will not work.

We apply Moser's iteration technique with adaptations to our periodic setting.  Set $\frac12 \le \sigma_2 < \sigma_1 \le 1$ and choose $\psi_1(r,\theta,z,s)=\phi_1(r)\eta_1(s)$ to be a smooth cut-off function defined on $P_1$ satisfying:
\begin{equation} \label{cutoff-phieta}
\begin{cases}
\text{supp}\,\phi \subset D_{\sigma_1}, \quad \phi=1\,\, \text{on}\,\, D_{\sigma_2}, \quad 0\le \phi \le 1, \\
\text{supp}\,\eta \subset (-(\sigma_1)^2,0], \quad \eta(s)=1 \,\, \text{on}\,\, (-(\sigma_2)^2,0],\quad 0 \le \eta\le 1, \\
|\eta'| \lesssim \frac{1}{(\sigma_1-\sigma_2)^2}, \quad |\nabla \phi| \lesssim \frac{1}{\sigma_1-\sigma_2}.
\end{cases}
\end{equation}
Consider the cut-off functions $\psi_R(x,s) = \phi_1(\frac{x}{R})\eta_1(\frac{s}{R^2})$. Testing \eqref{eq-Phi} by $\Phi \psi_R^2$ gives
\begin{align*}
-\frac{1}{2} \int_{-\infty}^t\int_{\mathbb{R}^3} \left( \partial_s \Phi^2 + (b\cdot \nabla) \Phi^2 + \frac{2}{r}\partial_r \Phi^2 \right) & \psi_R^2 dx ds \ge -\int_{-\infty}^t\int_{\mathbb{R}^3} (\Delta \Phi) \Phi \psi_R^2 dx ds \\
&= \int_{-\infty}^t\int_{\mathbb{R}^3} \left(|\nabla \Phi|^2 \psi_R^2 + \Phi \nabla \Phi \cdot \nabla \psi_R^2 \right) dx ds,
\end{align*}
for any $t\le 0$. Since
\begin{align*}
\int_{-\infty}^t\int_{\mathbb{R}^3} |\nabla \Phi|^2 \psi_R^2 dx ds \ge \int_{-\infty}^t\int_{\mathbb{R}^3} \left( \frac12 |\nabla(\Phi \psi_R)|^2 - \Phi^2 |\nabla \psi_R|^2 \right) dx ds,
\end{align*}
we get
\begin{align} \label{eq-21a}
\int (\Phi^2 \psi_R^2)(\cdot,t) dx + &\int_{-\infty}^t\int |\nabla (\Phi \psi_R)|^2 dxds  \le \int_{-\infty}^t\int -\left(  (b\cdot \nabla) \Phi^2 + \frac{2}{r}\partial_r \Phi^2 \right) \psi_R^2 dxds \\
&\quad+ \int_{-\infty}^t\int \left(\Phi^2 \partial_s \psi_R^2+2 \Phi^2 |\nabla \psi_R|^2 -2 \Phi \nabla\Phi \cdot \nabla \psi_R^2 \right) dxds. \nonumber
\end{align}
Now we treat the right hand side term by term. For the first term, we use $\nabla \cdot b =0$ to get
\begin{align} \label{eq-21b}
-\int_{-\infty}^t \int (b\cdot \nabla) \Phi^2 \psi_R^2 &= -\iint (v_r \partial_r \Phi^2 +v_z \partial_z \Phi^2) \psi_R^2 \nonumber \\
&=\iint (\partial_r v_r+ \frac{v_r}{r}+\partial_z v_z) \Phi^2\psi_R^2 + v_r \Phi^2 \partial_r \psi_R^2\nonumber\\
&= \iint v_r \Phi^2 \partial_r \psi_R^2.
\end{align} Here and later in the section, the integral element $dx ds$ is not written out unless there is confusion.
Let $L_\theta$ be the angular stream function which solves
\begin{align*}
\nabla \times (L_\theta e_\theta) = v_r e_r + v_z e_z,
\end{align*}
so that
\begin{align} \label{L-infty-(-1)}
v_r=-\partial_z L_\theta= \partial_z (L_\theta(r,z,t)-L_\theta(r,0,t)).
\end{align}
It is easy to check that $L_\theta$ is periodic with period $Z_0$, thus
$$|L_\theta(r,z,t)-L_\theta(r,0,t)| \le \sup|v_r(r,\cdot,t)| Z_0 \lesssim 1.$$
Hence we have
\begin{align} \label{eq-21c}
\iint v_r \Phi^2 \partial_r \psi_R^2 &= \iint -\left(L_\theta(r,z,t)-L_\theta(r,0,t)\right) \partial_z \Phi^2 \partial_r \psi_R^2\nonumber \\
&\le C \iint \Phi^2 (\partial_r \psi_R)^2 + \frac{1}{8}\iint (\partial_z \Phi)^2 \psi_R^2\nonumber \\
&\le C \iint_{P(\sigma_1 R)} \Phi^2 \frac{1}{(\sigma_2-\sigma_1)^2R^2} + \frac{1}{8}\iint (\partial_z (\Phi \psi_R))^2.
\end{align}
For the second term in \eqref{eq-21a}, using $\Phi\big|_{r=0}=0$ we get
\begin{align}\label{eq-21d}
-\iint \frac{2}{r} \partial_r \Phi^2 \psi_R^2 &= -\int_{-\infty}^t \int_0^{Z_0} \int_0^{2\pi} \int_0^{\infty} 2\partial_r \Phi^2 \psi_R^2 dr d\theta dz dt\nonumber\\
&= \int_{-\infty}^t \int_0^{Z_0} \int_0^{2\pi} \int_0^{\infty} 2 \Phi^2 \partial_r\psi_R^2 drd\theta dz dt\nonumber\\
&= \iint 2\Phi^2  \frac{\partial_r\psi_R^2}{r}\nonumber\\
&\le C \iint_{P(\sigma_1 R)} \Phi^2 \frac{1}{(\sigma_2-\sigma_1)R^2}.
\end{align}
The last three terms in (\ref{eq-21a}) are easier:
\begin{align} \label{eq-21e}
\iint \Phi^2 \partial_s \psi_R^2 + 2 \Phi^2 |\nabla \psi_R|^2 - 2\Phi \nabla\Phi \cdot \nabla \psi_R^2 &\le C \iint_{P(\sigma_1 R)} \Phi^2 \frac{1}{(\sigma_2-\sigma_1)^2R^2}\nonumber\\
&\quad + \frac{1}{8} \iint |\nabla (\Phi \psi_R)|^2.
\end{align}
Combing \eqref{eq-21a},\eqref{eq-21b},\eqref{eq-21c},\eqref{eq-21d},\eqref{eq-21e}, and using the properties of the cutoff functions \eqref{cutoff-phieta}, we arrive at
\begin{align}
\lab{energy}
\sup_{-(\sigma_2 R)^2 \le t\le 0} \|(\Phi\phi_R)(\cdot,t)\|_{L_x^2}^2 + \|\nabla (\Phi\psi_R)\|_{L_t^2 L_x^2}^2 \le C \iint_{P(\sigma_1 R)} \Phi^2 \frac{1}{(\sigma_1-\sigma_2)^2R^2}.
\end{align}

We have to use the following Sobolev embedding inequality for periodic functions:
\begin{align} \label{sobolev-embedding}
\|f\|_{L_x^3(D_{1})} \le C\|\nabla f\|_{L_x^2(D_{1})},
\end{align}
for any $f$ having period $Z_0$ in $z$ and compactly supported in the other two dimensions. To verify \eqref{sobolev-embedding}, one can argue as follows. Choose a cut-off function \begin{equation*}
g(z) = \begin{cases}
1, \quad 0<z\le NZ_0,\\
2-\frac{z}{NZ_0}, \quad NZ_0 \le z < 2NZ_0,\\
0, \quad \text{otherwise},
\end{cases}
\end{equation*}
with $N$ large. By the usual Sobolev embedding, after extending $f$ to the whole space in the periodic way along the $z$ axis, we deduce
\begin{align*}
\frac{N^\frac13}{2}\|f\|_{L^3(D_1)} \le \|fg\|_{L^3(R^3)} & \le C N^{\frac12} \|fg\|_{L^6(R^3)} \le C N^{\frac12}\|\nabla (fg)\|_{L^2(R^3)}\\
&\le C N^{\frac12}\| (\nabla f)g\|_{L^2} + CN^{\frac12} \|f(\partial_z g)\|_{L^2}\\
&\le C N\| \nabla f\|_{L^2(D_1)} + C \|f\|_{L^2(D_1)},
\end{align*}
which clearly implies \eqref{sobolev-embedding}. By scaling argument on the $x_1$ and $x_2$ direction, we have
\begin{align*}
R^{-\frac{2}{3}} \|(\Phi\psi_R)(\cdot,t)\|_{L_x^3(D_{R})} \le C\|\nabla (\Phi\psi_R)\|_{L_x^2(D_{R})}.
\end{align*}
We emphasize here that $R$ should be bounded from below, say by $1$ e.g. Interpolation of (\ref{energy}) gives
$$\left(\frac{1}{R^4} \iint_{P(\sigma_2 R)} \Phi^{\frac{5}{2}} \right)^\frac{2}{5} \le \frac{C}{\sigma_1-\sigma_2} \left(\frac{1}{R^4} \iint_{P(\sigma_1 R)} \Phi^2\right)^\frac{1}{2},$$
where $C$ does not depend on $R$.
Observe that $\Phi^{(\frac54)^k}, \quad k \ge 1$ are also positive subsolutions to \eqref{Gamma/vtheta}. Hence one can clearly repeat the above estimates to derive
$$\left(\frac{1}{R^4} \iint_{P(\sigma_{2k} R)} \Phi^{2\times(\frac54)^{k+1}} \right)^\frac{2}{5} \le \frac{C}{\sigma_{1k}-\sigma_{2k}} \left(\frac{1}{R^4} \iint_{P(\sigma_{1k} R)} \Phi^{2\times(\frac54)^k}\right)^\frac{1}{2},$$
for any $\frac{1}{2}\le\sigma_{2k}<\sigma_{1k}\le 1$. This is equivalent to
$$\left(\frac{1}{R^4} \iint_{P(\sigma_{2k} R)} \Phi^{2\times(\frac54)^{k+1}} \right)^{\frac12\times (\frac45)^{k+1}} \le \left(\frac{C}{\sigma_{1k}-\sigma_{2k}}\right)^{(\frac45)^k} \left(\frac{1}{R^4} \iint_{P(\sigma_{1k} R)} \Phi^{2\times(\frac54)^k}\right)^{\frac12\times (\frac45)^k}. $$
It remains to choose $\sigma_{1k}$ and $\sigma_{2k}$ converging to $\frac{1}{2}$ and iterate the above inequalities. This process is standard, thus omitted.
\end{proof}

\section{Liouville theorem for $\Gamma$}
In this section, we give the proof of Theorem \ref{thzhouqi} following the ideas in \cite{LZ11}, \cite{CSTY1} and \cite{Z}. Without loss of generality, let us assume $Z_0=1$ for simplicity in this section.
Assume that $\Phi$ is a positive periodic solution to \eqref{Gamma/vtheta} in $P_R=D_R \times (-R^2,0]$. We also assume that $\Phi\big|_{r=0} \ge \frac12$.

We denote $\Psi= -\ln \Phi$. Let us first carry out a standard analysis on $\Psi$, using $\nabla \cdot b =0$ and $v_r \in (L^\infty)^{-1}$. The equation for $\Psi$ reads
\begin{equation} \label{eq-Psi}
\partial_t \Psi + b \cdot \nabla \Psi + \frac{2}{r} \partial_r \Psi - \Delta \Psi + |\nabla \Psi|^2=0.
\end{equation}
Choose cut-off functions $\zeta_R(r,\theta,z)=\zeta_1(\frac{r}{R})$ such that
\begin{equation} \label{cutoff-zeta}
\begin{cases}
\zeta_R=1,\quad \text{for}\,\, x\in D_{R/2},\\
|\partial_r \zeta_R| \lesssim \frac{1}{R},\,\, \partial_\theta \zeta_R=\partial_z \zeta_R=0.
\end{cases}
\end{equation}
By testing \eqref{eq-Psi} with $\zeta_R^2$ and integrating on $D_R$ we get
\begin{align*}
\partial_t \int_{D_R} \Psi \zeta_R^2 dx + \int_{D_R} |\nabla \Psi|^2 \zeta_R^2 dx &= \int_{D_R} -b \cdot \nabla \Psi \zeta_R^2 - \frac{2}{r} \partial_r \Psi \zeta_R^2 - \nabla \Psi \cdot \nabla \zeta_R^2 \,\,dx\\
&\le \int_{D_R} -b \cdot \nabla \Psi \zeta_R^2 - \frac{2}{r} \partial_r \Psi \zeta_R^2 + \frac{1}{6}|\nabla \Psi|^2 \zeta_R^2 + C |\nabla \zeta_R|^2 \,\,dx \nonumber.
\end{align*}
Using \eqref{cutoff-zeta} we have
\begin{align}\label{eq-3a}
\partial_t \int \Psi \zeta_R^2 dx + \frac{5}{6}\int |\nabla \Psi|^2 \zeta_R^2 dx & \le C + \int \left(-b \cdot \nabla \Psi - \frac{2}{r} \partial_r \Psi \right)\zeta_R^2 dx.
\end{align}
The drift term involving $b$ can be estimated in the spirit of \eqref{eq-21b} and \eqref{eq-21c},
\begin{align}\label{eq-3b}
\int -b \cdot \nabla \Psi \zeta_R^2 &= \int (v_r \partial_r \zeta_R^2+ v_z \partial_z \zeta_R^2) \Psi\nonumber \\
&=-\int \partial_z (L_\theta(r,z,t)-L_\theta(r,0,t)) \partial_r \zeta_R^2 \Psi\nonumber\\
&=\int (L_\theta(r,z,t)-L_\theta(r,0,t)) \partial_r \zeta_R^2 \partial_z \Psi\nonumber\\
&\le C+\frac{1}{6} \int |\nabla \Psi|^2 \zeta_R^2.
\end{align}  Here we just used $| D_R | \sim R^2$ for large $R$.

To proceed, we need the weighted Poincar\'e inequality in our periodic domain $D_R\,(R\ge 1)$
\begin{align} \label{weighted-poincare}
\int_{D_R} |\Psi-\bar{\Psi}|^2\zeta_R^2 dx \le C R^2 \int_{D_R} |\nabla \Psi|^2 \zeta_R^2 dx,
\end{align}
where
\[
\bar{\Psi} = \left( \int \zeta_R^2 dx \right)^{-1} \int \Psi\zeta_R^2 dx.
\]To check this we first use the usual weighted Poincar\'e inequality in 2 dimensions to deduce
\begin{align*}
\int_{D_R} |\Psi-[\Psi](z)|^2\zeta_R^2 dx \le C R^2 \int_{D_R} |\nabla \Psi|^2 \zeta_R^2 dx,
\end{align*}
where
\[
[\Psi](z)=\left( \iint \zeta_R^2 rdr d\theta \right)^{-1}\iint \Psi \zeta_R^2 rdr d\theta.
\] Moreover, since $[\Psi]$ depends only on $z$ and $\bar{\Psi} = Z^{-1}_0 \int^{Z_0}_0 [\Psi](z) dz$, we have
\begin{align*}
\int_{D_R} |[\Psi]-\bar{\Psi}|^2 \zeta_R^2 dx &\le C R^2 \int_0^{Z_0} |[\Psi]-\bar{\Psi}|^2 dz\\
 &\le C R^2 \left(\int_0^{Z_0} |\partial_z [\Psi]| dz \right)^2 \\
&\le \frac{C}{R^2} \left(\int_{D_R} |\partial_z \Psi| \zeta_R^2 dx\right)^2 \\
&\le C \int_{D_R} |\partial_z \Psi|^2 \zeta_R^2 dx.
\end{align*} Here we have used a 1 dimensional Sobolev imbedding going from line 1 to line 2.
This proves \eqref{weighted-poincare}. Now integration by parts and \eqref{weighted-poincare} give
\begin{align} \label{eq-3c}
-\int \frac{2}{r} \partial_r \Psi \zeta_R^2 dx &= - 2\iiint \partial_r \Psi \zeta_R^2 dr d\theta dz \nonumber\\
&= 2\iint (\Psi-\bar{\Psi}) \zeta_R^2 d\theta dz \big|_{r=0}   + 2 \iiint (\Psi-\bar{\Psi}) \partial_r \zeta_R^2 dr d\theta dz\nonumber\\
&= C-C \bar{\Psi} +2\int (\Psi-\bar{\Psi}) \frac{\partial_r \zeta_R^2}{r} dx\nonumber\\
&\le C-C \bar{\Psi}+ \frac{1}{6} \int |\nabla \Psi|^2 \zeta_R^2 + C R^2 \int \left(\frac{\partial_r \zeta_R}{r}\right)^2 dx \nonumber\\
&\le C-C \bar{\Psi}+ \frac{1}{6} \int |\nabla \Psi|^2 \zeta_R^2 dx.
\end{align}
Hence, from \eqref{eq-3a},\eqref{eq-3b},\eqref{eq-3c},  we get a crucial differential inequality:
\begin{align} \label{inequality-Psi}
\partial_t \int \Psi \zeta_R^2 dx + C_1 \bar{\Psi} \le -\frac{1}{2} \int |\nabla \Psi|^2 \zeta_R^2 dx + C_2,
\end{align}
for $t\in [-R^2,0]$ and $C_1,\, C_2>0$ independent of $R$. At this point, we claim that the following lemma holds, since the same arguments in \cite{LZ11} apply to our situation with some adjustments on the region of integration.

\begin{lem} \label{lemma-31}
Let $\Phi \le 1$ be a positive z-periodic solution to \eqref{Gamma/vtheta} in $P_R ( R\ge 1)$ which satisfies
\begin{align} \label{lower-bound-l1}
\|\Phi\|_{L^1(P(\frac{R}{2}))}\ge \kappa R^4,
\end{align}
for some $\kappa >0$. Moreover we assume that $\Phi\big|_{r=0}\ge \frac{1}{2}$. Then there holds
\begin{align} \label{lemma-31-conclusion}
-\int \zeta_R^2(x) \ln \Phi(x,t) dx \le M R^2,
\end{align}
for all $t\in [-\frac{\kappa R^2}{4},0]$ and some positive constant $M$ depending only on $\kappa$.
\end{lem}

\begin{proof}

For the sake of completeness, we present the proof here. Note that
\[
d\mu=\frac{1}{R^2} (\int \zeta_1^2 dx)^{-1} \zeta_R^2 dx
\] is a probability measure.
By Nash's inequality(see Lemma \ref{lemma-nash} below) and the weighted Poincar\'e inequality \eqref{weighted-poincare}, and since $\Psi = - \ln \Phi$,
\begin{align}\label{eq-31a}
\left|\ln\left(\int_{D_R}\Phi d\mu\right) + \int_{D_R} \Psi d\mu \right|^2 \left(\int_{D_R} \Phi d\mu\right)^2 &\le |\sup \Phi|^2  \int_{D_R} \left|\Psi -\bar{\Psi}\right|^2 d\mu\nonumber\\
&\le C_3 \int_{D_R} |\nabla \Psi|^2 \zeta_R^2 dx.
\end{align}
For simplicity we write $a=\int \zeta_1^2 dx>0$. Plugging \eqref{eq-31a} into \eqref{inequality-Psi} gives
\begin{align}\label{eq-31b}
aR^2 \partial_t \bar{\Psi}(t) + C_1 \bar{\Psi}(t) \le C_2 - \frac{1}{2C_3} \left|\ln\int_{D_R}\Phi d\mu + \int_{D_R} \Psi d\mu \right|^2 \left(\int_{D_R} \Phi d\mu\right)^2.
\end{align}
Now we consider the set
\begin{align*}
W=\{s\in [-\frac14 R^2,0]:\int_{D_\frac{R}{2}} \Phi(s) dx \ge \frac{\kappa}{2} R^2\},
\end{align*}
and denote its characteristic function by $\chi(s)$. Due to the condition \eqref{lower-bound-l1}, we have
\begin{align*}
\kappa R^4 \le \int_{P_{R/2}} \Phi dx dt &= \int_W \int_{D_{R/2}} \Phi(s) dx ds + \int_{[-R^2/4,0]- W} \int_{D_{R/2}} \Phi(s) dx ds \\
&\le |W||D_{R/2}| \sup_{D_{R/2}}|\Phi|  + \frac{R^2}{4} \frac{\kappa R^2}{2}\\
&\le \frac{\pi R^2}{4} |W| + \frac{\kappa R^4}{8}.
\end{align*}
This gives
\begin{align} \label{lower-bound-|W|}
|W| \ge \frac{\kappa R^2}{2}.
\end{align}
From $aR^2 \partial_t \bar{\Psi} + C_1\bar{\Psi} \le C_2$, it is easy to derive that
\begin{align} \label{eq-31c}
\bar{\Psi}(s_2) \le \bar{\Psi}(s_1) + \frac{C_2}{C_1},
\end{align}
for any $-\frac{R^2}{4}\le s_1\le s_2 \le 0$. If for some $\frac{-R^2}{4} \le s\le \frac{-\kappa R^2}{4}$, there holds
$$\bar{\Psi}(s) \le 2\left|\ln \frac{\kappa}{2a}\right|+\frac{8a\sqrt{C_2C_3}}{\kappa},$$
then due to \eqref{eq-31c}, the conclusion \eqref{lemma-31-conclusion} holds with
\begin{equation*}
M=a\left(2\left|\ln \frac{\kappa}{2a}\right|+\frac{8a\sqrt{C_2C_3}}{\kappa} + \frac{C_2}{C_1}\right).
\end{equation*}
Otherwise, for all $-\frac{R^2}{4}\le s_1\le s_2 \le 0$ we have
$$\bar{\Psi}(s) \ge 2\left|\ln \frac{\kappa}{2a}\right| +\frac{8a\sqrt{C_2C_3}}{\kappa}.$$
Since for $s\in W\cap [-\frac{R^2}{4},-\frac{\kappa R^2}{4}]$, one has
$$\ln \int_{D_R} \Phi(s) d\mu \ge \ln \int_{D_{R/2}} \Phi(s) d\mu \ge \ln\frac{\kappa}{2a}.$$
In this case, \eqref{eq-31b} and $\Psi \ge 0$ gives
\begin{equation} \label{eq-31d}
aR^2 \partial_t \bar{\Psi}(t) \le a R^2 \partial_t \bar{\Psi}(t) + C_1 \bar{\Psi}(t) \le - C_4 \chi(t) \bar{\Psi}(t)^2,
\end{equation}
for $t \in [-\frac{R^2}{4},-\kappa\frac{R^2}{4}]$. Note that \eqref{lower-bound-|W|} implies
\begin{align*}
\int_{-R^2/4}^{-\kappa R^2/4} \chi(s) ds \ge \frac{\kappa R^2}{4}.
\end{align*}
Solving the Riccati type equation \eqref{eq-31d} clearly gives an absolute upper bound
for $\bar{\Psi}(-\frac{\kappa R^2}{4})$. See \cite{LZ11} Lemma 3.2 for details. The conclusion \eqref{lemma-31-conclusion} follows immediately by \eqref{eq-31c}.
\end{proof}

The Nash inequality used earlier can be found in \cite{CSTY2}. We give an easier proof here.
\begin{lem}\label{lemma-nash}
Let $\mu$ be a probability measure. Then for any integrable function $\Phi >0$ we have
\begin{equation} \label{nash-inequality}
\left| \ln \left( \int \Phi d\mu \right) - \int \ln \Phi d\mu \right| \left(\int \Phi d\mu\right) \le (\sup \Phi) \int \left|\ln \Phi-\int \ln \Phi d\mu \right| d\mu.
\end{equation}
\end{lem}
\begin{proof}
After multiplying $\Phi$ by a constant, one may assume that $\int \ln \Phi d\mu=0$. In this case, Jensen's inequality gives
\begin{equation*}
\ln \int \Phi d\mu \ge \int \ln \Phi d\mu = 0.
\end{equation*}
For the convex function $f(\alpha)=\alpha\ln \alpha$, using Jensen's inequality again we get
\begin{align*}
 \ln  \left(\int \Phi d\mu \right)  \left(\int \Phi d\mu\right) \le \int \Phi \ln \Phi d\mu \\
 \le (\sup \Phi) \int |\ln \Phi| d\mu,
\end{align*}
which proves \eqref{nash-inequality}.
\end{proof}

We shall need another auxiliary lemma giving a lower bound of $\int_{P_R} \Phi\, dx dt$, which makes Lemma \ref{lemma-31} applicable.

\begin{lem}\label{lemma-32}
Let $\Phi$ be a nonnegative z-periodic solution(with period $Z_0=1$ in the $z$ direction) to \eqref{Gamma/vtheta} in $P_R\,\, (R\ge 1)$, satisfying
$$\Phi\big|_{r=0} \ge \frac{1}{2}.$$
Then
\begin{equation}\label{lemma-32-conclusion}
\|\Phi \|_{L^1(P_R)} \ge \kappa R^4,
\end{equation}
for some absolute constant $\kappa > 0$ independent of R.
\end{lem}
\begin{proof}
Consider cut-off functions $\psi=\psi_R(x,t)$ compactly supported on $P_R$, satisfying
\begin{equation} \label{cutoff-psi}
\begin{cases}
\psi_R = 1, \quad \text{for}\quad (x,t) \in D_{R/2} \times [-\frac34 R^2, -\frac14 R^2], \\
\partial_z \psi_R =0 , \quad |\nabla \psi_R| \lesssim \frac{1}{R},\\
|\partial_t \psi_R|,\,\, |\nabla^2 \psi_R| \lesssim \frac{1}{R^2}.
\end{cases}
\end{equation} For simplicity of presentation, we will drop the index $R$ in $\psi_R$ unless stated otherwise.
Let us test \eqref{Gamma/vtheta} by $\frac{1}{2\sqrt{\Phi}}\psi_R^2$ in the domain $P_R$:
\begin{align}\label{eq-32a}
\int_{P_R}  -\sqrt{\Phi} \partial_t \psi^2 + \frac{2}{r} \partial_r(\sqrt{\Phi}) \psi^2 + b \cdot \nabla\sqrt{\Phi} \psi^2&= \int_{P_R}  \Delta \Phi \frac{\psi^2}{2\sqrt{\Phi}}\nonumber\\
&= \int_{P_R}  \sqrt{\Phi} \Delta \psi^2 + 4\int_{P_R}  |\nabla (\Phi^\frac14)|^2 \psi^2.
\end{align}
The singular drift term can be estimated similarly as before,
\begin{align} \label{eq-32b}
\int \frac{2}{r} \partial_r \sqrt{\Phi} \psi^2 &= -2 \iiint \sqrt{\Phi}\big|_{r=0} \psi^2 d\theta dz dt - 2 \int  \sqrt{\Phi} \frac{\partial_r \psi^2}{r}\nonumber\\
&\le - \kappa_1 R^2 -2\int \sqrt{\Phi} \frac{\partial_r \psi^2}{r},
\end{align}
where $\kappa_1$ is a positive constant. Then we again use $\nabla \cdot b =0$ and $v_r = - \partial_z(L_\theta-L_\theta(r,0,t))$ to get
\begin{align} \label{eq-32c}
\int b \cdot \nabla\sqrt{\Phi} \psi^2 &= -\int v_r \sqrt{\Phi} \partial_r \psi^2 = \int (L_\theta-L_\theta(r,0,t)) \partial_z \sqrt{\Phi} \partial_r \psi^2 \nonumber\\
&\le \int |\nabla (\Phi^\frac14)|^2 \psi^2 + C\int \sqrt{\Phi} (\partial_r \psi)^2.
\end{align}
We plug \eqref{eq-32b}, \eqref{eq-32c} into \eqref{eq-32a} to get
\begin{align*}
\int_{P_R} \sqrt{\Phi} (-\partial_t \psi^2 -2\frac{\partial_r \psi^2}{r} + C (\partial_r \psi)^2 -\Delta \psi^2) \ge \kappa_1 R^2.
\end{align*}
Due to \eqref{cutoff-psi} we have
\begin{align*}
\int_{P_R} \sqrt{\Phi} \frac{1}{R^2} \ge \kappa_2 R^2,
\end{align*}
for some positive constant $\kappa_2$ independent of $R$. It remains to conclude \eqref{lemma-32-conclusion} using H\"older's inequality.
\end{proof}

Now we are ready to give:
\begin{proof}[Proof of Theorem \ref{thzhouqi}]
Let us work with $|\Gamma|\le 1$ and $Z_0=1$. Since the Liouville theorem for no swirl case has been proved in \cite{KNSS}, it suffices to prove that $\Gamma \equiv 0$. Consider the domain $P_R=D_R\times (-R^2,0]$ for some $R\ge 1$.

We may assume that
$$\sup_{P_R} \Gamma \le -\inf_{P_R} \Gamma.$$
Otherwise consider $-\Gamma$. Let
$$\Phi = \frac{\Gamma - \inf_{P_R} \Gamma}{\sup_{P_R} \Gamma - \inf_{P_R} \Gamma}.$$
Then $0\le \Phi \le 1$ and $\Phi\big|_{r=0}\ge \frac12$. By Lemma \ref{lemma-31} and Lemma \ref{lemma-32} we deduce that for all $t\in [-\frac{\kappa R^2}{4},0],$
$$-\int \zeta_R^2(x) \ln \Phi(x,t) dx \le M R^2.$$
By Chebyshev's inequality, for any $0<\delta<1$ and $t\in [-\frac{\kappa R^2}{4},0]$,
\begin{equation} \label{eq-11a}
|\{x\in D_{R/2}: \Phi(x,t)\le \delta\}| \le \frac{M R^2}{|\ln \delta|}.
\end{equation}
Since $(\delta - \Phi)_+$ is a nonnegative Lipschitz subsolution to \eqref{Gamma/vtheta}, we apply Lemma \ref{lemma-21} and use \eqref{eq-11a} to deduce
\begin{align*}
\sup_{P_{\sqrt{\kappa}R/4}} (\delta-\Phi)_+ &\lesssim \left\{\frac{1}{R^4} \iint_{P_{\sqrt{\kappa}R/2}} (\delta-\Phi)_+^2 dx dt\right\}^\frac{1}{2} \\
&\le \left\{\frac{M\delta^2}{R^2|\ln \delta|}\right\}^\frac12.
\end{align*}
Choose a $\delta$ small enough we get a point-wise lower bound
$$\Phi(x,t) \ge \frac{\delta}{2},$$
for $(x,t) \in P_{\sqrt{\kappa}R/4}$. This implies
$$\left(\sup_{P_{\sqrt{\kappa}R/4}}- \inf_{P_{\sqrt{\kappa}R/4}}\right) \Phi \le 1-\sigma,$$
for some constant $\sigma >0$. Hence
\begin{equation} \label{eq-11b}
\left(\sup_{P_{\sqrt{\kappa}R/4}}- \inf_{P_{\sqrt{\kappa}R/4}}\right) \Gamma \le (1-\sigma)\left(\sup_{P_{R}}- \inf_{P_{R}}\right) \Gamma.
\end{equation}
Iterating \eqref{eq-11b} for a sequence of $R_k \to \infty$, we get $\Gamma \equiv \Gamma(x=0,t=0)=0$. As mentioned earlier, this implies $v = c e_z$ with $c$ being a constant.
\end{proof}

\section{Proof of Theorem \ref{thgudai}}

In this section, we will use a weighted energy method to treat non-periodic ancient solutions under an extra assumption on the convergence rate of $\Gamma$.
We will prove that
\be
\lab{gammato0}
\lim\sup_{r \to \infty} \Gamma = 0
\ee uniformly in $t$ and $z$. Then \cite{LZZ} Theorem 1.3, Remark 1.4, will imply that $v= c e_z$.
Recall the equation for $\Gamma$:
\be
\lab{eqgammat}
\p^2_r \Gamma + \p^2_z \Gamma - v_r \p_r \Gamma  - v_z \p_z \Gamma - \frac{1}{r} \p_r \Gamma - \p_t \Gamma= 0.
\ee

Suppose for contradiction that (\ref{gammato0}) is false.

Then
\[
\limsup_{r \to \infty, z, t \to t_\infty} |\Gamma| =c_0 \neq 0
\] for some constant $c_0$. Here $t_\infty$ is either $-\infty$ or a finite negative number.
  We can assume,
without loss of generality, that $c_0=1$, otherwise we can multiply $\Gamma$ by a suitable constant.

 First, we make the observation that
 \be
 \lab{gam<1}
 |\Gamma| \le 1.
 \ee The reason is that
\be \label{limsup=sup}
\limsup_{r \to \infty, z, t \to t_\infty} |\Gamma| = \sup_{r, z, t} |\Gamma|.
\ee
Otherwise there would be a bounded sequence $\{ r_i \}$, $z_i$ and $t_i \to t_\infty$ such that
\[
\lim_{i \to \infty} |\Gamma(r_i, z_i, t_i)| = \sup_{r, z, t} |\Gamma|.
\] Consider the translated sequence
 \[
 \Gamma_i=\Gamma_i(r, z, t)=\Gamma(r, z, t+t_i).
 \]Then we can find a subsequence of $\Gamma_i$ which converges, in $C^{2, 1}_{loc}$ topology,  to $\Gamma_\infty$ which is a bounded ancient solution of
\[
\Delta \Gamma_\infty - \bar{b} \nabla \Gamma_\infty - \frac{2}{r} \p_r \Gamma_\infty - \p_t \Gamma_\infty=0.
\]Here $\bar b$ is a bounded $C^{2, 1}$ vector field. We can suppose that $r_i \to r_\infty<\infty$ and $z_i \to z_\infty$. Then $\Gamma_\infty$ reaches nonzero interior maximum away from the $z$ axis at the point $(r_\infty, z_\infty, 0)$. Hence $\Gamma_\infty$ is a nonzero constant by the maximum principle. This contradicts with the fact that  $\Gamma_\infty=0$ at the $z$ axis. This proves (\ref{gam<1}).
So for the rest of the proof we can and do assume that $\lim\sup_{r \to \infty} |\Gamma|
 =\sup |\Gamma| =1$.

Let $R>R_0>r_0(>1)$ be large numbers which can be chosen later. In this section, we take a weight function $\lam$ to be
\be
\lab{lamf2}
\al
\lam=
\begin{cases}  r, &\qquad r \in [0, r_0];\\
r_0-\frac{r_0-1}{R_0-r_0} (r-r_0) &\qquad r \in [r_0, R_0];\\
1, &\qquad r \ge R_0.
\end{cases}
\eal
\ee
Instead of $rdrdz$, we will perform energy estimates against the measure $\lambda(r)drdz$ with carefully chosen cut-off functions.

We consider the domain
\[
D = D_1 \cup D_2,
\]where \[ D_1=\{(r, z, t) \, | \, 0 \le r <R_0;  -R \le z \le R,
-T \le t \le 0 \},
\]
\[
\quad D_2 =\{(r, z, t) \, | \, R_0 \le r  \le R;  -R \le z \le R, -T \le t \le 0  \}.
\]

Let
\[
\phi_1=\xi_1(z) \frac{t}{-T}
\]where $\xi=\xi(z)$ is a smooth cut-off function on $[-R, R]$ such that $0 \le \xi_1 \le 1$,
$|\xi'_1| \le C/R$ and $\xi_1(z)=1$ when $z \in [-R/2, R/2]$.

Let $\phi_2=\phi_2(r, z, t)$ be the unique solution to the final time boundary value problem of the backward equation, which is well-posed.
\be
\lab{eqphi2.t4}
\al
\begin{cases}
&\p^2_r \phi_2 + \p^2_z \phi_2 +
 \frac{2 \lam'(r)}{\lam(r)} \p_r \phi_2+ v_r \p_r \phi_2 + v_z \p_z \phi_2
 \\
 &\qquad + \frac{1}{2} \left[ \left(\frac{ \lam'(r)}{\lam(r)}-\frac{1}{r} \right) v_r + (\frac{\lam(r)}{r})' \frac{1}{\lam(r)} \right]  \phi_2 -A \phi_2 +\p_t \phi_2 = 0, \quad (r, z) \in D_2,\\
&\phi_2(R_0, z, t) =\phi_1(z, t), \quad \phi_2(R, z, t)=0, t \in [-T, 0]; \quad \phi_2=0 \quad \text{if} \quad
z=R, -R; \\
&\phi_2(r, z, 0)=0.
\end{cases}
\eal
\ee Here we take $A= \frac{1}{2 R_0} \Vert v_r \Vert_\infty$. We keep $\lambda$ in the expressions for possible future use, even though $\lambda=1$ for $r \ge R_0$.

 Since $r \ge R_0>1$ and $\lam=1$ on $D_2$, standard parabolic equation theory tells us that the above problem has a unique solution such that $0 \le \phi_2 \le 1$ and $\Vert \nabla \phi_2 \Vert_\infty \le A_0$. Here $A_0$ is a constant depending only on $C^1$ norm of $b=v_r e_r + v_z e_z$.

 Now we compute
 \[
 \al
 &-\int_{D_1} ( \p^2_r \Gamma + \p^2_z \Gamma ) \Gamma \phi^2_1 d\mu dt\\
 & =
 -\int_{D_1} ( \p^2_r \Gamma + \p^2_z \Gamma ) \Gamma \phi^2_1 \lam(r) drdzdt\\
 &=\int_{D_1} ( |\p_r \Gamma|^2 + |\p_z \Gamma|^2 ) \phi^2_1 d\mu dt +
 \int^0_{-T}\int^R_{-R} \int^{R_0}_0 (\p_r \Gamma) \Gamma \phi^2_1(z, t) \lam'(r) drdzdt \\
 &+ \int^0_{-T}\int^R_{-R} \int^{R_0}_0 (\p_z \Gamma) \Gamma 2 \phi_1(z, t) \p_z \phi_1(z, t)  \lam(r) drdzdt
 -\int^0_{-T}\int^R_{-R} (\p_r \Gamma) \Gamma \phi^2_1(z, t) \lam(r) \big|_{r=R_0} dz.
 \eal
 \]Similarly,
 \[
 \al
 -\int_{D_2} &( \p^2_r \Gamma + \p^2_z \Gamma ) \Gamma \phi^2_2 d\mu dt\\
 &=\int_{D_2} ( |\p_r \Gamma|^2 + |\p_z \Gamma|^2 ) \phi^2_2 d\mu +
 \int^0_{-T}\int^R_{-R} \int^R_{R_0} (\p_r \Gamma) \Gamma \phi^2_2 \lam'(r) drdzdt \\
 &\qquad + \int^0_{-T}\int^R_{-R} \int^R_{R_0} (\p_r \Gamma) \Gamma 2 \phi_2 \p_r \phi_2  \lam(r) drdzdt
 +\int^0_{-T}\int^R_{-R} \int^R_{R_0} (\p_z \Gamma) \Gamma 2 \phi_2 \p_z \phi_2  \lam(r) drdzdt\\
 &\qquad +\int^0_{-T}\int^R_{-R} (\p_r \Gamma) \Gamma \phi^2_2 \lam(r) \big|_{r=R_0} dzdt.
 \eal
 \]
Adding the previous two identities, noting the last two boundary terms cancel, and also $\lam'=0$ when $r >R_0$, we obtain
\be
\lab{Lt74}
\al
L &\equiv \int_{D_1} ( |\p_r \Gamma|^2 + |\p_z \Gamma|^2 ) \phi^2_1 d\mu dt +
\int_{D_2} ( |\p_r \Gamma|^2 + |\p_z \Gamma|^2 ) \phi^2_2 d\mu dt\\
&= \underbrace{-\int_{D_1} ( \p^2_r \Gamma + \p^2_z \Gamma ) \Gamma \phi^2_1 d\mu dt
}_{T_1} \quad \underbrace{- \int^0_{-T}\int^R_{-R} \int^{R_0}_0 (\p_r \Gamma) \Gamma \phi^2_1(z, t) \lam'(r) drdzdt}_{T_2} \\
 &\qquad \underbrace{- \int^0_{-T}\int^R_{-R} \int^{R_0}_0 (\p_z \Gamma) \Gamma 2 \phi_1 \p_z \phi_1  \lam(r) drdzdt }_{T_3} \quad
 \underbrace{-\int_{D_2} ( \p^2_r \Gamma + \p^2_z \Gamma ) \Gamma \phi^2_2 d\mu dt}_{T_4}\\
 &\qquad
 \underbrace{ - \int^0_{-T} \int^R_{-R} \int^R_{R_0} (\p_r \Gamma) \Gamma 2 \phi_2 \p_r \phi_2  \lam(r) drdzdt}_{T_5} \quad
\underbrace{ -\int^0_{-T}\int^R_{-R} \int^R_{R_0} (\p_z \Gamma) \Gamma 2 \phi_2 \p_z \phi_2  \lam(r) drdzdt}_{T_6}\\
&\equiv T_1 + ... +T_6.
\eal
\ee In the rest of the proof we will find an upper bound for each $T_i$, $i=1, ..., 6$.

{\it Step 3. bound for $T_1$.}

From equation (\ref{eqgammat}),
\[
T_1 = - \int^0_{-T}\int^R_{-R} \int^{R_0}_0 ( v_r \p_r \Gamma + v_z \p_z \Gamma + \frac{1}{r} \p_r \Gamma
+\p_t \Gamma)
\Gamma \phi^2_1 \lam(r) drdzdt.
\]After integration by parts and using the divergence free property of $v_r e_r + v_z e_z$, we
deduce
\be
\lab{t1t164}
\al
T_1 &= -\frac{1}{2} \int^0_{-T}\int^R_{-R} \int^{R_0}_0 \left[ v_r \p_r (\Gamma^2-1)
+ v_z \p_z (\Gamma^2-1) + \frac{1}{r} \p_r \Gamma^2   + \p_t (\Gamma^2 -1) \right]
 \phi^2_1(z, t) \lam(r) drdzdt\\
 &=\underbrace{\frac{1}{2} \int^0_{-T}\int^R_{-R} \int^{R_0}_0 ( v_r \p_r \phi^2_1
+ v_z \p_z \phi^2_1 ) (\Gamma^2-1) \lam(r) drdzdt}_{T_{11}} \\
 &\qquad \underbrace{- \frac{1}{2} \int^0_{-T}\int^R_{-R} v_r (\Gamma^2-1) \phi^2_1 \lam(r) \big|_{r=R_0} dzdt}_{T_{12}}\\
 &\qquad \underbrace{+\frac{1}{2} \int^0_{-T}\int^R_{-R} \int^{R_0}_0 v_r (\Gamma^2-1) \phi^2_1 \left( \lam'(r)- \frac{\lam(r)}{r} \right) drdzdt}_{T_{13}}\\
 &\qquad \underbrace{+\frac{1}{2} \int^0_{-T}\int^R_{-R} \int^{R_0}_0 \Gamma^2 \phi^2_1 \left( \frac{\lam(r)}{r} \right)' drdzdt}_{T_{14}} \quad \underbrace{- \frac{1}{2} \int^0_{-T}\int^R_{-R}  \Gamma^2 \phi^2_1 \frac{\lam(r)}{r} \big|_{r=R_0} dzdt}_{T_{15}}\\
 &\qquad
  + \underbrace{\frac{1}{2} \int^0_{-T}\int^R_{-R}\int^{R_0}_0  (\Gamma^2-1) \p_t \phi^2_1 \lam(r) drdzdt}_{T_{16}} +  \underbrace{ \left(-
  \frac{1}{2} \right) \int^R_{-R}\int^{R_0}_0 (\Gamma^2-1) \phi^2_1 \big|^{t=0}_{t=-T} \lam(r) drdz}_{T_{17}}\\
 &\equiv T_{11} + ... + T_{17}.
\eal
\ee

Notice that
\[
|T_{11}| \le \Vert v_z \Vert_\infty \frac{C T}{R} R R_0 \Vert \lam \Vert_\infty \Vert \Gamma^2-1 \Vert_\infty = C R_0 r_0 T \Vert v_z \Vert_\infty \, \Vert \Gamma^2-1 \Vert_\infty.
\] The term $T_{12}$ is a  boundary one, which will be cancelled with a boundary term from integration on $D_2$, called $T_{42}$. $T_{14} \le 0$ and $T_{15} \le 0$. Also, since
\[
|\p_t \phi_1| \le 1/T
\] and
\[
\int^{R_0}_0 \lam(r) dr = \frac{1}{2} (R_0-r_0) + \frac{1}{2} R_0 r_0,
\] direct computation shows
\be
\lab{t164}
T_{16} \le \frac{1}{2} \Vert \Gamma^2-1 \Vert_\infty \,  R (R_0 r_0 + R_0 - r_0) \le \Vert \Gamma^2-1 \Vert_\infty \,  R R_0 r_0.
\ee
The second inequality is due to $r_0 > 1$. Similarly,
\be
\lab{t174}
T_{17} \le -\frac{1}{2} \inf [1-\Gamma^2(R_0, \cdot, -T)] \,  R (R_0 r_0 + R_0 - r_0) \le 0.
\ee Hence, we can deduce, after leaving $T_{12}$ and $T_{13}$ along for now, that
\be
\lab{t1<t4}
\al
T_1 &\le C R_0 r_0 \Vert v_z \Vert_\infty  \Vert \Gamma^2-1 \Vert_\infty T  +  \Vert \Gamma^2-1 \Vert_\infty  R R_0 r_0
+T_{12} \\
 &\qquad +\frac{1}{2} \int^0_{-T} \int^R_{-R} \int^{R_0}_0 v_r (\Gamma^2-1) \phi^2_1 \left( \lam'(r)- \frac{\lam(r)}{r} \right) dr dz dt.
\eal
\ee

{\it Step 4. bounds for $T_2, ..., T_7$}

First we bound $T_2$. By our choice of $\lam$,
\be
\lab{t2<t}
\al
T_2&= -\frac{1}{2} \int^0_{-T}\int^R_{-R} \int^{R_0}_0 \p_r (\Gamma^2-1) \phi^2_1 \lam'(r) dr dzdt\\
&=-\frac{1}{2} \int^0_{-T}\int^R_{-R} \phi^2_1 \int^{r_0}_0 \p_r (\Gamma^2-1)   drdzdt
+\frac{r_0-1}{2(R_0-r_0)} \int^0_{-T}\int^R_{-R} \phi^2_1 \int^{R_0}_{r_0} \p_r (\Gamma^2-1)   drdzdt\\
&\le \left[-\frac{1}{6}  + \left(\frac{1}{6}+ \frac{r_0-1}{3(R_0-r_0)} \right) \sup_{r = r_0, t \in [-T, 0]} (1-\Gamma^2) \right] R T.
\eal
\ee

Next
\be
\lab{t3<t}
T_3 \le C \Vert \p_z \Gamma \Vert_\infty r_0 R_0 T
\ee since $|\p_z \phi_1| \le C/R$.

Now we deal with the terms $T_4, ..., T_7$, which involve integrations on $D_2$ only.
From equation (\ref{eqgammat}),
\[
T_4 = - \int^0_{-T}\int^R_{-R} \int^R_{R_0} ( v_r \p_r \Gamma + v_z \p_z \Gamma + \frac{1}{r} \p_r \Gamma + \p_t \Gamma)
\Gamma \phi^2_2 \lam(r) drdzdt.
\]After integration by parts and using the divergence free property of $v_r e_r + v_z e_z$, we
deduce
\be
\lab{t4=t}
\al
T_4 &= -\frac{1}{2} \int^0_{-T}\int^R_{-R} \int^R_{R_0} \left[ v_r \p_r (\Gamma^2-1)
+ v_z \p_z (\Gamma^2-1) + \frac{1}{r} \p_r (\Gamma^2-1) + \p_t (\Gamma^2-1)\right]
 \phi^2_2 \lam(r) drdzdt\\
 &=\underbrace{ \frac{1}{2} \int^0_{-T}\int^R_{-R} \int^R_{R_0} ( v_r \p_r \phi^2_2
+ v_z \p_z \phi^2_2 ) (\Gamma^2-1) \lam(r) drdzdt}_{T_{41}} \\
 &\qquad \underbrace{+ \frac{1}{2} \int^0_{-T}\int^R_{-R} v_r (\Gamma^2-1) \phi^2_2 \lam(r) \big|_{r=R_0} dz}_{T_{42}}\\
 &\qquad \underbrace{+\frac{1}{2} \int^0_{-T}\int^R_{-R} \int^R_{R_0} v_r (\Gamma^2-1) \phi^2_2 \left( \lam'(r)- \frac{\lam(r)}{r} \right) drdzdt}_{T_{43}}\\
 &\qquad \underbrace{+\frac{1}{2} \int^0_{-T}\int^R_{-R} \int^R_{R_0} (\Gamma^2-1) \phi^2_2 \left( \frac{\lam(r)}{r} \right)' drdzdt}_{T_{44}} \quad
  \underbrace{+ \frac{1}{2} \int^0_{-T}\int^R_{-R}  (\Gamma^2-1) \phi^2_2 \frac{\lam(r)}{r} \big|_{r=R_0} dzdt}_{T_{45}}\\
  &\qquad \underbrace{+ \frac{1}{2} \int^0_{-T}\int^R_{-R} \int^R_{R_0} (\Gamma^2-1) \p_t \phi^2_2\lam(r) drdzdt}_{T_{46}} \quad + \underbrace{\frac{1}{2} \int^R_{-R} \int^R_{R_0} (\Gamma^2-1)  \phi^2_2\lam(r) \big|_{t=-T} drdz}_{T_{47}}\\
 &\equiv T_{41} + ... + T_{47}.
\eal
\ee Notice that $T_{42}$ will cancel with $T_{12}$ when all terms are added. Also $T_{45} \le 0$
and
\be
\lab{t47}
T_{47} \le -  \inf_{r \in [R_0, R]} (1 -\Gamma^2(r, \cdot, -T)) \int^R_{-R} \int^R_{R_0}
\phi^2_2(r, z, -T) \lam(r)  drdz \le 0.
\ee Therefore
\[
T_4 \le T_{41} + T_{42} + T_{43} + T_{44} + T_{46}.
\]

Next
\be
\lab{t5<t}
\al
T_5 &= -\frac{1}{2} \int^0_{-T}\int^R_{-R} \int^R_{R_0} \p_r (\Gamma^2-1) \p_r \phi^2_2 \lam(r) drdzdt\\
&=\frac{1}{2} \int^0_{-T}\int^R_{-R} \int^R_{R_0}  (\Gamma^2-1) \p^2_r \phi^2_2 \lam(r) drdzdt - \frac{1}{2} \int^0_{-T}\int^R_{-R}  (\Gamma^2-1) \p_r \phi^2_2 \lam(r) \big|^R_{R_0} dzdt,
\eal
\ee since $\lam(r)=1$ here. Finally
\be
\lab{t6<t}
T_6 = \frac{1}{2} \int^0_{-T}\int^R_{-R} \int^R_{R_0}  (\Gamma^2-1) \p^2_z \phi^2_2 \lam(r) drdzdt.
\ee

Combining the bounds on $T_i$, $i=1, ..., 6$, noticing cancellation of boundary terms
$T_{12}$  with $T_{42}$,
we find that
\[
\al
L &\le \left[-\frac{1}{6}  + \left(\frac{1}{6}+ \frac{r_0-1}{3(R_0-r_0)} \right) \sup_{r = r_0, t \in [-T, 0]} (1-\Gamma^2) \right] R T\\
  &\qquad + C R_0 r_0 \Vert v_z \Vert_\infty \Vert \Gamma^2- 1 \Vert_\infty T  + \Vert \Gamma^2-1 \Vert_\infty R R_0 r_0 + C \Vert \p_z \Gamma \Vert_\infty r_0 R_0 T \\
&\qquad+\frac{1}{2} \int^0_{-T}\int^R_{-R} \int^{R_0}_0 v_r (\Gamma^2-1) \phi^2_1(z) \left( \lam'(r)- \frac{\lam(r)}{r} \right) dr dz dt\\
&\qquad +\frac{1}{2} \int^0_{-T}\int^R_{-R} \int^R_{R_0} \bigg[\p^2_r \phi^2_2 + \p^2_z \phi^2_2 + v_r \p_r \phi^2_2
+ v_z \p_z \phi^2_2 + 2 \frac{\lam'(r)}{\lam(r)} \p_r \phi^2_2\\
&\hskip3cm +
v_r  \left( \frac{\lam'(r)}{\lam(r)}- \frac{1}{r} \right) \phi^2_2
 +\left( \frac{\lam(r)}{r} \right)' \frac{1}{\lam(r)} \phi^2_2 + \p_t \phi^2_2\bigg] (\Gamma^2-1) \lam(r) drdzdt \\
&\qquad + \frac{1}{2} \int^0_{-T}\int^R_{-R}  (\Gamma^2-1) \p_r \phi^2_2 \lam(r) \big|_{r=R_0} dzdt.
\eal
\]Since $\Gamma^2- 1 \le 0$ and $\phi_2$ is a solution to (\ref{eqphi2.t4}), the second from last integral in the above inequality is nonpositive. Hence
\be
\lab{l<34}
\al
L &\le  -\frac{1}{6} R T + \left(\frac{1}{6}+ \frac{r_0-1}{3(R_0-r_0)} \right) \sup_{r = r_0, t \in [-T, 0]} (1-\Gamma^2) R T  + C R_0 r_0 \Vert v_z \Vert_\infty  T \, \Vert \Gamma^2-1 \Vert_\infty\\
 &\qquad+ \Vert \Gamma^2-1 \Vert_\infty R R_0 r_0+ C \Vert \p_z \Gamma \Vert_\infty r_0 R_0 T\\
  &\qquad+\underbrace{\frac{1}{2} \int^0_{-T}\int^R_{-R} \int^{R_0}_0 v_r (\Gamma^2-1) \phi^2_1(z) \left( \lam'(r)- \frac{\lam(r)}{r} \right) drdzdt}_{I_1} \\ &\qquad  + \underbrace{\frac{1}{2} \int^0_{-T}\int^R_{-R}  (\Gamma^2-1) \p_r \phi^2_2 \lam(r) \big|_{r=R_0} dzdt}_{I_2}\\
&\equiv ...+ I_1 + I_2.
\eal
\ee Here the last two integrals are denoted by $I_1$, $I_2$ respectively.

{\it Step 5.}

It remains to bound the two integrals $I_1$ and $I_2$.

First let us bound $I_2$. Observe that the coefficients of lower order terms of the equation (\ref{eqphi2.t4}) are bounded by
 $\frac{1}{R_0}+ \Vert v_r \Vert_\infty + \Vert v_z \Vert_\infty$ in $D_2$ and that the boundary value of $\phi_2$ at $r=R_0$ is $\phi_1$ which satisfies $0 \le \phi_1 \le 1$, $|\p_z \phi_1|  \le C/R$
 and $|\p_t \phi_1| \le 1/T$. By standard boundary gradient bound for parabolic equations, we know that
 \[
 |\p_r \phi_2|_{r=R_0} \le C_1, \qquad C_1=C_1(\Vert v_r \Vert_\infty, \Vert v_z \Vert_\infty).
 \]Hence
 \be
 \lab{i2<t4}
 I_2 \le C_1 \sup_{r=R_0, t \in [-T, 0]} (1-\Gamma^2)  R T.
 \ee

Finally
\[
I_1 = \frac{1}{2} \int^0_{-T}\int^R_{-R} \int^{R_0}_{r_0} v_r (\Gamma^2-1) \phi^2_1(z, t) \left( \lam'(r)- \frac{\lam(r)}{r} \right) drdzdt.
\]By direct computation, we see that
\[
\left| \lam'(r)- \frac{\lam(r)}{r} \right| \le 2 \frac{r_0}{r}
\]when $R_0 >> r_0$. Therefore, under the assumption of the theorem, i.e.
\[
|\Gamma^2(r, z, t)-1| \le \e/r,
\] we have
\[
|I_1| \le C  R T \sup_{r_0 \le r \le R_0} |v_r| \, r_0 \int^{R_0}_{r_0} \frac{\e}{r^2} dr \le
 C \e R T \sup_{r_0 \le r \le R_0} |v_r|.
\]

\be
\lab{i1<t4}
I_1 \le C \e R T.
\ee

Now we substitute (\ref{i2<t4}) and (\ref{i1<t4}) into (\ref{l<34}) to obtain
\be
\lab{ineq421}
\al
L &\le
-\frac{1}{6} R T + \left(\frac{1}{6}+ \frac{r_0-1}{3(R_0-r_0)} \right) \sup_{r = r_0, t \in [-T, 0]} (1-\Gamma^2) R T  + C R_0 r_0 \Vert v_z \Vert_\infty  T \Vert \Gamma^2-1 \Vert_\infty\\
 & \qquad +  \Vert \Gamma^2-1 \Vert_\infty  R R_0 r_0+ C \Vert \p_z \Gamma \Vert_\infty r_0 R_0 T\\
  &\qquad  + C_1 \sup_{r = R_0,  t \in [-T, 0]} (1-\Gamma^2) R T + C \e R T.
\eal
\ee
Recalling the definition of $L$ from the first line of (\ref{Lt74}), and using $\Gamma^2 \to 1$, we deduce
\be
\lab{mainineq2.14}
\al
& \int^0_{-T}\int_{D_1} ( |\p_r \Gamma|^2 + |\p_z \Gamma|^2 ) \phi^2_1 d\mu +
\int^0_{-T}\int_{D_2} ( |\p_r \Gamma|^2 + |\p_z \Gamma|^2 ) \phi^2_2 d\mu \\
&\le -\frac{1}{12} R T   + C \Vert \Gamma^2-1 \Vert_\infty  R_0 r_0 \Vert v_z \Vert_\infty  T \\
 &\qquad +  \Vert \Gamma^2-1 \Vert_\infty R R_0 r_0 + C \Vert \p_z \Gamma \Vert_\infty r_0 R_0 T\\
  &\qquad   + C \e R T,
\eal
\ee when $R_0 >> r_0 >> 1.$

{\it Step 6.}

Note that $|\partial_z \Gamma|$ is uniformly bounded by standard parabolic theory. Under the assumptions  of the theorem, inequality (\ref{mainineq2.14}) is impossible when $R$ is sufficiently large and $T >> R_0 r_0$, $R_0>>r_0$ and
$\e$ is sufficiently small.
Hence (\ref{gammato0}) is true, i.e. $\lim_{r \to \infty} \Gamma =0$ uniformly. As proven in \cite{LZZ} Theorem 1.3, Remark 1.4, this shows $v_\theta=0$ and $v=c e_z$. \qed

\section{Steady periodic solutions}

In this section, we demonstrate that the weighted energy estimate method used in Section 4 can, in fact, be adapted to give an elementary proof of Theorem \ref{thzhouqi} with an extra assumption: $v$ is stationary in time. First, we recall the following general fact:

\begin{lem}[Sliding  Property]
\label{lem-sliding-time}
Let $v$ be a bounded ancient mild solution to ASNS with $|\Gamma| \lesssim 1$. Let $(x_n, t_n)$ be any sequence with $x_n = (r_n, 0, z_n)$ such that $r_n \to \infty$. Then, up to a subsequence,
$v$ uniformly converges to a constant vector on the parabolic cube $Q_R(x_n, t_n)=\{(x, t) \,  | \, |x-x_n|<R, \, 0<t_n-t<R^2 \}$ for any given $R > 0$. Moreover, up to a further subsequence, $\Gamma$ uniformly converges to a constant on $Q_R(x_n, t_n)$.
\end{lem}
\begin{proof}
The conclusion is known in the literature. See the proof of Theorem 1.1 in \cite{LZ}.
 But for completeness, let's give a simple proof here. Define
$$v^{(n)}(x, t) = v(x_n + x, t_n + t),\quad p^{(n)}(x, t) = p(x_n + x, t_n+t).$$ Here $p$ is the pressure.
Clearly, $(v^{(n)}, p^{(n)})$ is still a bounded ancient mild solution to the Navier-Stokes equations, so is its weak limit $(v^{(\infty)}, p^{(\infty)})$ (up to a subsequence). Moreover, the convergence from $(v^{(n)}, p^{(n)})$ to its limit $(v^{(\infty)}, p^{(\infty)})$ is locally strong in $C^{2k, k}_{{\rm loc}}$ for any $k \geq 0$.

Now consider any parabolic cube $Q_R = B_R(0) \times [-R^2, 0]$. Up to subsequence, we have
\begin{equation}\nonumber
\begin{cases}
e_r(x + x_n) \to e_1,\quad e_\theta(x + x_n) \to e_2,\\[-4mm]\\
v^{(n)}(x, t)\cdot e_\theta(x + x_n) = \frac{\sqrt{y_1^2 + y_2^2} v(y, t_n + t)\cdot e_\theta(y)}{\sqrt{y_1^2 + y_2^2}}\Big|_{y = x_n + x} \rightarrow 0,
\end{cases}
\quad (x, t) \in Q_R,
\end{equation} where $y=x_n +x$.
where $e_1$, $e_2$ are two perpendicular vectors. Hence, we have
$$v^{(\infty)}\cdot e_2 = 0\quad {\rm on}\ Q_R.$$
 Write $x=x_1 e_1 + x_2 e_2 + x_3 e_z$ with $e_1, e_2$ given above and $e_z=(0, 0, 1)$.
  Due to
\begin{eqnarray}\nonumber
\partial_{x_2}v^{(n)}(x, t) &=& \partial_{x_2}v(x_n + x, t_n+t) = \partial_{y_2}v(y, t_n+t)\\\nonumber
&=& \frac{y_2}{r(y)}\partial_r v(y, t_n + t) + \frac{y_1}{r(y)^2}\partial_\theta v(y, t_n + t)\\\nonumber
&=& \frac{y_2}{r(y)}\partial_r v(y, t_n + t)  + \frac{y_1}{r(y)^2} v_r(y, t_n + t)e_\theta(y)\\\nonumber
&& -\ \frac{y_1}{r(y)^2} v_\theta(y, t_n + t) e_r(y) \rightarrow 0\ \ {\rm on}\ Q_R,
\end{eqnarray}
we conclude that $v^{(\infty)}$ on $Q_R$ is independent of $x_2$. Here $r(y)$ is the distance from $y$ to the $z$ axis. Since $R$ is arbitrary, we can conclude that $$(v^{(\infty)}\cdot e_1, v^{(\infty)}\cdot e_z)$$
is a bounded ancient mild solution to the  2 dimensional Navier-Stokes equations. Using the Liouville theorem in \cite{KNSS}, we see that $v^{(\infty)}$ must be a constant vector (independent of time). This proves the convergence for $v$.

Next we turn to $\Gamma$.
Define
$$\Gamma^{(n)}(x, t) = \Gamma(x_n+x, t_n + t).$$
Up to a further subsequence, one has
$$\Gamma^{(n)} \to \Gamma^{(\infty)},\quad |\Gamma^{(\infty)}| \lesssim 1,$$
and the convergence is locally strong in $C^{k, 2k}_{{\rm loc}}$ for any $k \geq 0$. Moreover, one can similarly derive that $\Gamma^{(\infty)}$ is independent of $x_2$. From the equation
\begin{eqnarray}\label{G-eqn}
\partial_t\Gamma + b\cdot\nabla\Gamma + \frac{1}{r}\partial_r\Gamma = (\partial_r^2 + \partial_z^2)\Gamma,
\end{eqnarray} we deduce
$$\partial_t\Gamma^{(n)} + v^{(n)}\cdot\nabla\Gamma^{(n)} + \frac{y}{r^2}\Big|_{y = x_n + x}\cdot(\p_1, \p_2, 0)^T\Gamma^{(n)} = \Delta\Gamma^{(n)},$$
one sees that
$$\partial_t\Gamma^{(\infty)} + v^{(\infty)}\cdot\nabla\Gamma^{(\infty)} = (\partial_1^2 + \partial_z^2)\Gamma^{(\infty)}.$$ Note that $v^\infty$ is a constant vector. Therefore, one can convert the above equation into the standard heat equation by a change of variable.
The standard Liouville theorem for the heat equation implies that $\Gamma^{(\infty)}$ must be a constant.  We have proved the lemma.
\end{proof}

Next we prove some useful convergence properties of $v_r$ and $\Gamma$, concerning ancient $z$-periodic solutions.

\begin{lem}\label{lem-vvanishing}
Let $v$ be a bounded mild ancient solution to the ASNS. If $v$ is periodic in $z$, then $v_r \to 0$ uniformly for $(z, t) \in \mathbb{T}^1 \times (-\infty, 0)$ as $r \to \infty$. If $v$ is furthermore steady in time, then there exists a constant $c$ such that $\Gamma  \to c$ uniformly for $z \in \mathbb{T}^1$.
\end{lem}
\begin{proof}
Suppose the conclusion is not true, then there exists $c_0 \neq 0$ and a sequence of points $P_n=(r_n, 0 , z_n, t_n)$ with $r_n \to \infty$ such that
$$\lim_{n \to \infty}v_r(P_n) = c_0.$$
Here $0$ is the angle in the cylindrical system.

Using the Sliding property in Lemma \ref{lem-sliding-time}, one sees that the solution $u$ converges to a constant  on the parabolic ball $Q_R(P_n)$ with any given radius $R$, centering at $P_n$. This means that
\begin{equation}\label{e1-1}
|v_r| \geq \frac{|c_0|}{2}\quad {\rm on}\ Q_R(P_n)
\end{equation}
as $n$ is large enough.

Now let $L_\theta$ be the angular stream function which solves
$$\nabla\times (L_\theta e_\theta) = v_r  e_r + u_z e_z,$$
which gives
$$v_r = - \partial_z L_\theta.$$
Since $L_\theta$ is periodic in $z$, there exists $z = z(r, t)$ such that
$$v_r(r, z(r, t), t) = 0.$$
for any $r$ and $t$. This clearly contradicts \eqref{e1-1}. Hence the conclusion of the lemma is true.

Next we consider the steady case and prove convergence of $\Gamma$. Pick a sequence $r_i \to 0$. By Lemma \ref{lem-sliding-time}, we can find a subsequence, still denoted by $r_i$ such that $\Gamma(r_i, z)$ converges uniformly to a constant $c_0$.
Hence for any $\e>0$, there exists integer $N>0$, if $i \ge N$, then $|\Gamma(r_i, z)-c_0| < \e$ for all $z, t$. Pick any $r>r_N$. Then there exists integers $i, j >N$ such that
$r_i \le r \le r_j$. In the domain $[r_i, r_j] \times [-Z_0, Z_0]$ for $(r, z)$, the maximum principle for $\Gamma$ gives
\[
c_0-\e \le \Gamma(r, z) \le c_0+\e.
\]
This proves the lemma. \end{proof}

From now on, we work on a bounded ancient solution $v$ to the ASNS, periodic in $z$ and independent of time $t$. Assuming that $\sup |\Gamma|=1$, one can argue as in Section 4 to perform a weighted energy estimate. Since $v$ is a steady solution, we do not need to integrate in time. Take
$\lambda$ as in Section 4, $\phi_1=\xi_1(z),$
and let $\phi_2(r,z)$ solve the elliptic problem
\be
\al
\begin{cases}
&\p^2_r \phi_2 + \p^2_z \phi_2 +
 \frac{2 \lam'(r)}{\lam(r)} \p_r \phi_2+ v_r \p_r \phi_2 + v_z \p_z \phi_2
 \\
 &\qquad + \frac{1}{2} \left[ \left(\frac{ \lam'(r)}{\lam(r)}-\frac{1}{r} \right) v_r + (\frac{\lam(r)}{r})' \frac{1}{\lam(r)} \right]  \phi_2 -A \phi_2 = 0, \quad (r, z) \in D_2,\\
&\phi_2(R_0, z) =\phi_1(z), \quad \phi_2(R,z)=0 \\
& \phi_2=0 \quad \text{if} \quad
z=R, -R;
\end{cases}
\eal
\ee
Note that $\phi_2$ still statisfies a standard boundary gradient bound $|\partial_r \phi_2|_{r=R_0} \le C_1(\|v\|_{L^\infty}).$
The terms $T_{16}, T_{17}, T_{46}, T_{47}$, coming from time derivatives, will not appear now. Thus we arrive at(See the derivation of \eqref{ineq421})
\be \label{mainineq-last}
\al
& \int_{D_1} ( |\p_r \Gamma|^2 + |\p_z \Gamma|^2 ) \phi^2_1 d\mu +
\int_{D_2} ( |\p_r \Gamma|^2 + |\p_z \Gamma|^2 ) \phi^2_2 d\mu \\
&\le -\frac{1}{6} R  + \left(\frac{1}{6}+ \frac{r_0-1}{3(R_0-r_0)} \right) \sup_{r = r_0} (1-\Gamma^2) R    + C \Vert \Gamma^2-1 \Vert_\infty  R_0 r_0 \Vert v_z \Vert_\infty   \\
 &\qquad  + C \Vert \p_z \Gamma \Vert_\infty r_0 R_0  + C_1 \sup_{r = R_0} (1-\Gamma^2) R \\
  &\qquad + |I_1|.
  \eal
\ee
Here the integral $I_1$ reads
\[
I_1 = \frac{1}{2} \int^R_{-R} \int^{R_0}_{r_0} v_r (\Gamma^2-1) \phi^2_1(z) \left( \lam'(r)- \frac{\lam(r)}{r} \right) drdz.
\]
This estimate is rather general, since we have not used the periodicity in $z$, or any convergence infomation on $\Gamma$. By Lemma \ref{lem-vvanishing} and the observation \eqref{limsup=sup}, we have $v_r \to 0$ and $\Gamma \to 1$ when $r\to \infty$. We continue to treat $I_1$ in a different way from Section 4.
Using $v_r = \partial_z L_\theta$ and integration by parts, we have
\begin{align}\label{I1new}
I_1 &= -\frac{1}{2} \int_{-R}^R \int_{r_0}^{R_0} (L_\theta(r,z)-L_\theta(r,0)) \partial_z (\Gamma^2-1) \phi_1^2(z) \left( \lam'(r)- \frac{\lam(r)}{r} \right) dr dz \nonumber\\
&\quad -\frac{1}{2}  \int_{-R}^R \int_{r_0}^{R_0} (L_\theta(r,z)-L_\theta(r,0))  (\Gamma^2-1) \partial_z \phi_1^2(z) \left( \lam'(r)- \frac{\lam(r)}{r} \right) dr dz \nonumber\\
&=: W_1 + W_2 \nonumber
\end{align}
Note that
$$|L_\theta(r,z)-L_\theta(r,0)| \le \sup_z |v_r(r,z)| Z_0 = o(1) \quad \text{as} \quad r\to \infty$$
and
$$\left| \lambda'(r)-\frac{ \lambda(r) }{r} \right| \le \frac{r_0-1}{R_0-r_0}+1 $$
for $r_0\le r \le R_0$. Since $\Gamma \to 1$ and $|\partial_z \phi_1| \lesssim \frac{1}{R}$, we have
\be \label{w2}
|W_2| \le o(1)(R_0-r_0)\left(\frac{r_0-1}{R_0-r_0}+1\right) \le o(1) R_0.
\ee
Here $o(1)$ is defined when $r_0 \to \infty$. Using H\"older's inequality,
\begin{align} \label{w1}
|W_1| &\le \frac{1}{2}  \int_{-R}^R \int_{r_0}^{R_0} |\partial_z \Gamma|^2 \phi_1^2 \lambda(r)drdz \nonumber\\
&\quad + 2 \int_{-R}^R \int_{r_0}^{R_0} (L_\theta(r,z)-L_\theta(r,0))^2 \Gamma^2 \phi_1^2 \left( \lambda'(r)-\frac{ \lambda(r)}{r}\right)^2 \frac{1}{\lambda(r)} drdz \nonumber \\
&\le \frac{1}{2}  \int_{-R}^R \int_{r_0}^{R_0} |\partial_z \Gamma|^2 \phi_1^2 \lambda(r)drdz\nonumber \\
&\quad + o(1) \frac{(r_0-1)^2}{R_0-r_0} R + o(1)R \int_{r_0}^{R_0} \frac{\lambda(r)}{r^2} dr\nonumber \\
&\le \frac{1}{2}  \int_{-R}^R \int_{r_0}^{R_0} |\partial_z \Gamma|^2 \phi_1^2 \lambda(r)drdz \nonumber \\
&\quad + o(1) \frac{(r_0-1)^2}{R_0-r_0} R + o(1)R.
\end{align}
Combing \eqref{mainineq-last}, \eqref{w1} and \eqref{w2} we finally arrive at
\be
\al
& \frac{1}{2}\int_{D_1} ( |\p_r \Gamma|^2 + |\p_z \Gamma|^2 ) \phi^2_1 d\mu +
\int_{D_2} ( |\p_r \Gamma|^2 + |\p_z \Gamma|^2 ) \phi^2_2 d\mu \\
&\le -\frac{1}{6} R  + \left(\frac{1}{6}+ \frac{r_0-1}{3(R_0-r_0)} \right) \sup_{r = r_0} (1-\Gamma^2) R    + C \Vert \Gamma^2-1 \Vert_\infty  R_0 r_0 \Vert v_z \Vert_\infty   \\
 &\qquad  + C \Vert \p_z \Gamma \Vert_\infty r_0 R_0  + C_1 \sup_{r = R_0} (1-\Gamma^2) R \\
  &\qquad + o(1) \frac{(r_0-1)^2}{R_0-r_0} R + o(1)R.
  \eal
\ee

It remains to let $r_0>>1$, and $R_0 >> r_0$ to obtain a contradiction. Thus we have proved $\Gamma \equiv 0$, which leads to the conclusion of Theorem \ref{thzhouqi} in the steady case.
\section*{Acknowledgement}
Z.L. was supported by NSFC (grant No. 11725102) and National Support Program for Young Top-Notch Talents. Q.S.Z. wishes to thank the Simons Foundation for its support, and Fudan University for its hospitality during his visit.


\frenchspacing
\bibliographystyle{plain}

\end{document}